\documentclass{amsart}

\usepackage{amsmath, amsthm, amssymb}
\usepackage{epsfig}
\usepackage{array}
\usepackage{amsmath}
\usepackage{amsfonts}
\usepackage{amsfonts}
\usepackage{amsmath}%
\usepackage{amsfonts}%
\usepackage{amssymb}%
\usepackage{graphicx}

\usepackage{dsfont}

\usepackage{amsfonts}
\usepackage{amssymb}
\usepackage{amsmath,amsxtra,amssymb}
\usepackage{color}

\numberwithin{equation}{section}

\newtheorem{theorem}{Theorem}[section]
\newtheorem{lemma}[theorem]{Lemma}
\newtheorem{remark}[theorem]{Remark}

\newcommand{\om}{{\omega}}

\title
{A note on concentration functions on transformation groups}
\author{Mehrdad Kalantar \and Mohammad S. M. Moakhar}

\address{ School of Mathematics and Statistics,
        Carleton University, Ottawa, Ontario, Canada K1S 5B6}
\email{mkalanta@math.carleton.ca}

\address{Department of Mathematics,
        Tarbiat Modares University, Tehran 14115-134, Iran}
\email{m.mojahedi@modares.ac.ir}


\begin{document}
\maketitle

\begin{abstract}
In this note, we consider the concentration function problem for a continuous action of a locally compact group $G$ on a locally compact Hausdorff space $X$.
We prove a necessary and sufficient condition for the concentration functions of a
spread-out irreducible probability measure $\mu$ on $G$ to converge to zero.
\end{abstract}

\section{Introduction}
Let $G$ be a locally compact group, and let $\mu$ be a probability measure on $G$. 
The {\it concentration function problem} concerns 
the conditions on $G$ or $\mu$ under which the sequence
\[
\mathcal{F}_n(K) \,=\, \sup_{x\in G}\,\mu^n (Kx^{-1})
\]
converge to zero for every compact set $K \subset G$.

In \cite{HM}, Hofmann--Mukherjea partially answered this problem; 
namely they proved that the above sequence converges to zero,
when $G$ is non-compact and $\mu$ is irreducible (i.e. the semigroup generated by the support of $\mu$ is dense in $G$).
They moreover conjectured it to be true
for all such $G$ and $\mu$. After the theory of totally disconnected
groups had been developed, in \cite{JRW}, Jaworski--Rosenblatt--Willis used the theory
to prove Hofmann--Mukherjea's conjecture in full generality.

This note is a result of our attempt to understand the main properties in the above setup that forces the convergence.
In order to single out those properties of the group
responsible for the result to hold, we consider the problem in a more general setting;
we consider the problem in the setting of continuous group actions.

Suppose $G$ acts on a locally compact Hausdorff space $X$ by homeomorphisms.
For $x\in X$ and compact $K\subseteq X$, denote $K_x := \{t \in G: t x \in K\}$.
We define concentration functions of a probability measure $\mu$ on $G$ by
\begin{equation}\label {c.f.}
\mathcal{F}_n(K) \,=\, \sup_{x\in X} \, \mu^n(K_x) \hspace{1cm} (K\subseteq X \text{ is compact })
\end{equation}

The main result of this paper proves that
when $\mu$ is irreducible and spread-out, the convergence of concentration functions of $\mu$ to zero is
equivalent to the lack of $\mu$-stationary measure on the space $X$.

Recall that the probability measure $\nu$ on the $G$-space $X$ is called $\mu$-stationary if for all $\phi\in C_c(X)$,
the space of continuous functions on $X$ with compact support,
\[
\int_{X}\, \int_{G} \,{\phi(t x) \,d\mu(t) \, d\nu(x)} \,=\, \int_{X}{\phi(x)d\nu(x)} \,.
\]
If $\mu$ is an irreducible probability measure on $G$, then $G$ has a $\mu$-stationary measure if and only if $G$ is compact.
Hence, our result really shows that non-existence of stationary measures is that property of non-compact locally compact groups
which yields to the convergence of concentration functions to zero.

Before stating our main result, let us recall some definitions.
In the following, for probability measures $\mu$ and $\nu$ on $G$, $\mu * \nu$ denotes the convolution measure,
which is the probability measure on $G$ determined by
\[
\int_G\, f(r)\,d\mu*\nu(r)\,=\,\int_G\,\int_G \, f(ts) \,d\mu(t)\,d\nu(s) \,\,\,\,\,\,\,\, (f\in C_c(G))\,.
\]
The $n$-th iterated convolution power $\mu * \cdots * \mu$ ($n$ times) is denoted by $\mu^n$. 
For $\phi\in C_b(X)$, the space of bounded continuous functions on $X$, we define the function $\mu * \phi \in C_b(X)$ by
\[
\mu * \phi (x) \,:=\, \int_G \, \phi(tx) \, d\mu(t)\,.
\]
Observe that $\|\mu * \phi\|_\infty \leq \|\phi\|_\infty$.
A probability measure $\mu$ on $G$ is said to be spread-out if some convolution power $\mu^n$ is nonsingular with repeat to 
the Haar measure.

The main result of this note is the following.
\begin{theorem}\label{main}
Let $G$ be a locally compact group, and let $\mu$ be an irreducible and spread-out probability measure on $G$.
Suppose $\alpha: G \curvearrowright X$ is
a continuous action of $G$ on a locally compact Hausdorff space $X$. Then the following are equivalent:
\begin{itemize}
\item [1.]
For every compact $K\subset X$,
\[
\mathcal{F}_n(K) \,\longrightarrow \,0 \,;
\]
\item [2.]
the space $X$ admits no $\mu$-stationary probability measure.
\end{itemize}
\end{theorem}

\section{Proof of the main result}

For the proof of our theorem we need the following lemma which we believe should be known to the experts.
 
\begin{lemma}\label{0-2}
Let $\mu$ be an irreducible and spread-out probability measure on the locally compact group $G$. 
Then there exists $k\in\mathbb{N}$ such that
\[
\lim_n \|\,\mu^{n+k} \, -\, \mu^n\,\|_{1} \, =\,0\,.
\]
\end{lemma}

\begin{proof}
Since $\mu$ is spread-out, there exists $m\in\mathbb{N}$ such that the $m$-th convolution power
$\mu^m$ can be decomposed as $\mu^m = f+\nu$, where $0\neq f\in L^1(G)^+$ and $\nu\in M(G)^+$.
Let $S \subseteq G$ be the support of $f$ (i.e. the complement of the union of
all open subsets of $G$ on which $f$ is almost everywhere zero).
Consider the continuous function $t \mapsto \int_S f(t^{-1}r)\  d\mu(r)$, since 
$\mu$ is irreducible there exists
$k\in\mathbb{N}$ such that
\[
\int_S \,\mu^k * f(r)\, d\mu(r) \,=\, \int_S \,\int_G\, f(t^{-1}r)\, d\mu^k(t)\, d\mu(r) \,=\, \int_G \,\int_S\, f(t^{-1}r)\, d\mu(r)\, d\mu^k(t)\, >\, 0\,.
\]
It follows that if $S'\subseteq G$ denotes the support of $\mu^k * f$, then $S\cap S' \neq \emptyset$.
This implies
\[
\|\, f\, -\, \mu^k * f\,\|_1 \,<\, \|\, f \,\|_1 \,+\, \|\,\mu^k * f\,\|_1 \,\leq \, 2 \,\|f\|_1 ,
\]
and therefore
\begin{eqnarray*}
\|\, \mu^m \,-\, \mu^{m+k}\,\|_1 
&=& \|\, f\,+\,\nu\, -\, (f+\nu) * \mu^{k}\,\|_1
 \\ &\leq &
\|\, f\, -\, f * \mu^{k}\,\|_1 \,+ \,\|\, \nu\, -\, \nu * \mu^{k}\,\|_1
\\&<&
 2 \,\|\,f\,\|_1 \,+\, 2\, \|\,\nu\,\|_1 \,=\, 2\,.
\end{eqnarray*}
Hence the result follows from Foguel's 0-2 law \cite[Theorem I]{Fog}.
\end{proof}

\noindent
{\it Proof of Theorem \ref{main}.} $(1) \Rightarrow (2):$ suppose for the sake of contradiction that $\lambda$
is a nonzero $\mu$-stationary measure on $X$. 
Since $\lambda$ is regular, we may find a compact subset $K\subseteq X$ such that $\lambda(K) > 0$. 
Moreover, since $\lambda$ is $\mu$-stationary we have
\[\begin{array}{lll}
\displaystyle
\mathcal{F}_n(K) \,=\, \sup_{x\in X} \, \mu^n(K_x) &\geq & \int_X \, \mu^n(K_x) \, d\lambda(x) \\
&=& \int_X\,\int_G \,\mathds{1}_{K}(tx) \, d\mu^n(t)\,d\lambda(x) \\ &=& 
\int_X\,\mathds{1}_{K}(x) \, d\lambda(x) \\ &=& \lambda(K)\,.
\end{array}\]
Hence the concentration functions $(\mathcal{F}_n)$ do not converge to zero.

\noindent
$(2) \Rightarrow (1):$ Fix a probability measure $\nu$ on $X$. For $\phi \in C_b(X)$ and $n\in\Bbb N$, set
\[
a^{(\phi )}_{n} \,:= \,\int_X \,\int_G \, \phi(tx) d\mu^n(t) \, d\nu(x) \,,
\]
and consider the sequence $(a^{(\phi)}_{n}) \in \ell^\infty(\Bbb N)$.
Note that $a^{(\mu * \phi)}_{n} = a^{(\phi )}_{n+1}$.

Now, let $F$ be a shift invariant positive linear functional on $\ell^\infty(\Bbb N)$ that
extend the limit (c.f. \cite[Theorem 7.1]{Co}).
Define a positive linear functional $\Lambda$ on $C_b(X)$ by
$\Lambda(\phi) \,=\, F(\,(a^{(\phi)}_{n})\,)$.

If $\Lambda$ is not zero on $C_c(X)$, then Riesz representation theorem implies that
there exists a probability measure $\lambda$ on $X$ such that
$\int_X \phi \,d\lambda = \Lambda(\phi)$ for $\phi\in C_c(X)$. Moreover, since
$F$ is shift invariant, it follows
\begin{eqnarray*}
\int_X\,\int_G \, \phi(t x) \, d\mu(t)\,d\lambda(x) &=& \Lambda (\mu * \phi) \,=\, F(\,(a^{(\mu * \phi)}_{n})\,) \\ &=& F(\,(a^{(\phi)}_{n+1})\,)
\,=\, F(\,(a^{(\phi)}_{n})\,) \\ &=& \Lambda(\phi) \,=\, \int_X \,\phi(x)\, d\lambda(x) ,
\end{eqnarray*}
which shows that $\lambda$ is $\mu$-stationary.
But since by the assumption $X$ does not admit a $\mu$-stationary probability measure,
we conclude that $\Lambda$ is zero on $C_c(X)$.

From the properties of $F$ (c.f. \cite[Theorem 7.1]{Co}), it then follows for $\phi\in C_c(X)^+$ that
\[
0 \,\leq\, \liminf_n\,{\int_X\int_G \phi(t x) \, d\mu^n(t)\, d\nu(x)}\, \leq\, F(\,(a^{(\phi)}_{n})\,) \,=\, \Lambda(\phi) \,=\, 0\,.
\]
This implies there is a subnet $(\mu^{n_i})$ of $(\mu^{n})$ such that
\begin{equation}\label{00}
\lim_i \, {\int_X \, \int_G \, \phi(t x) \, d\mu^{n_i}(t) \, d\nu(x)} \,=\, 0\,.
\end{equation}
Now, given any finite collections $\nu_1, \nu_2, \dots, \nu_{m_1} \in Prob(X)$ and $\phi_1, \phi_2, \dots, \phi_{m_2} \in C_c(X)^+$,
by applying \eqref{00} to $\displaystyle\nu = \frac{1}{m_1}\sum_{m=1}^{m_1} \nu_m$ and $\displaystyle \phi = \sum_{m'=1}^{m_2} \phi_{m'}$
we get
\[
\lim_i \, {\int_X \, \int_G \, \phi_{m'}(tx) \, d\mu^{n_i}(t) \, d\nu_m(x)} \,=\, 0
\]
for all $1\leq m \leq m_1$ and $1\leq m' \leq m_2$.
Hence, we can construct a subnet $(\mu^{n_j})$ of $(\mu^{n})$ such that (\ref{00}) holds for all $\phi\in C_c(X)$ and $\nu\in Prob(X)$.
Moreover, for $s\in\mathbb{N}$, replacing $\phi$ by $\mu^s * \phi$ in (\ref{00}), we conclude
\begin{equation}\label{01}
\lim_j\, {\int_X \, \int_G \, \phi(t x)\, d\mu^{n_j+s}(t)\, d\nu(x)} \,= \,0 
\end{equation}
for all $\phi\in C_c(X)$ and $\nu\in Prob(X)$.

Now fix $\psi\in C_c(X)^+$.  
Since the sequence $(\|\mu^n * \psi\|_\infty)$ is positive and decreasing, it has a limit.
We claim that the limit is zero.

To prove the claim, suppose for the sake of contradiction that $\|\mu^n * \psi\|_\infty > \alpha > 0$ for all $n\in\Bbb N$.
Then for every $n\in\Bbb N$ there is a probability measure $\nu_n\in Prob(X)$ such that
\[
\int_X \, \int_G \, \psi(t x)\, d\mu^{n}(t)\, d\nu_n(x) \,>\, \alpha \,.
\]
Since $\| \mu^n * \nu_n \|_1 \leq 1$, 
it follows from the Banach--Alaoglu Theorem there is a subnet $( \mu^{n_i} * \nu_{n_i} )$,
and a measure $\rho$ on $X$ such that
\[
{\int_X \, \phi(x)\, d\mu^{n_i} * \nu_{n_i}(x)} \,\longrightarrow \,\int_X \, \phi(x)\, d\rho(x)
\]
for all $\phi\in C_c(X)$.
Therefore, if we let $k\in\Bbb N$ be as in Lemma \ref{0-2}, we get 
\[  \begin{array}{lll}
\mid \int_X \,\int_G \, \phi(tx)\, d\mu^{k}(t)\,d\rho(x) - \int_X \, \phi(x)\,d\rho(x)\mid
& = & \lim_i \mid \int_X \, \phi(x)\, d\mu^{n_i+k} * \nu_{n_i}(x) - \int_X \, \phi(x)\, d\mu^{n_i} * \nu_{n_i}(x)\mid\\
& = & \lim_i \mid \int_X \, \phi(x)\, d\big [\mu^{n_i+k} * \nu_{n_i} - \mu^{n_i} * \nu_{n_i}\big ](x)\mid\\
& \leq & \|\phi\|_\infty \, \lim_i {\|\,\mu^{n_i+k}- \mu^{n_i}\, \|_{1}}.\\
& = & 0
\end{array} \]
for all $\phi\in C_c(X)$.
Hence for every $m\in \Bbb N$, 
\[\begin{array}{lll}
\int_X\, \int_G\, \psi(t x)\,d\mu^{mk}(t)\,d\rho(x) 
&=& 
\int_X\, \int_G\, \psi(t x)\,d\mu^{(m-1) k}(t)\,d \mu^k*\rho(x) 
\\ &=&
\int_X\, \int_G\, \psi(t x)\,d\mu^{(m-1) k}(t)\,d \rho(x) 
\\ &=&
\cdots
\\ &=&
\int_X\, \int_G\, \psi(t x)\,d\mu^{k}(t)\,d \rho(x) 
\\ &=&\int_X\, \psi(x)\,d\rho(x)
\\ &>& 
\alpha \,.
\end{array}\]
On the other hand, by (\ref{01}) we can find an $n_{j_0}$ large enough so that 
${\int_X \, \int_G \, \psi(t x)\, d\mu^{n_{j_0}+s}(t)\, d\nu(x)} \,<\, \alpha/2$,
for every $s = 1, ..., k$.
But $n_{j_0}+s_0 = km$ for some $1\leq s_0\leq k$ and $m\in\Bbb N$, and therefore
\[
\alpha \,<\, \int_X \,\int_G\, \psi(t x)\,d\mu^{mk}(t) \,d\rho(x) \,=\, \int_X \,\int_G\, \psi(t x)\,d\mu^{n_{j_0}+s_0}(x) \,d\rho(x) \,<\, {\alpha}/{2} \,.
\]
This contradiction yields the claim.

To finish the proof, take a compact $K\subset X$. Applying the Urysohn's Lemma,
we can construct $\phi\in C_c(X)$ such that $\phi = 1$ on $K$, and therefore
\[  \begin{array}{cll}
\displaystyle
\sup_{x\in X}\,\mu^n(K_x) & = & \displaystyle\sup_{x\in X}\,\int_G\,\mathds{1}_{K}(t x)\,d\mu^n(t) \\
& \leq &\displaystyle \sup_{x\in X}\,\int_G\,\phi(t x)\,d\mu^n(t)
\end{array} \] 
which goes to zero by the claim. \hfill $\qed$

\begin{remark}
\emph{If $G$ is compact, then every continuous action of $G$ on any locally compact space admits stationary probability measures.
In fact, let $\om$ be the Haar probability measure on $G$. Then it is easily seen that for any $\nu\in Prob(X)$, the convolution $\om * \nu$
is $\mu$-stationary measure for all $\mu\in Prob(G)$.}

\emph{On the other hand, if $X$ is compact, it is well-known that
any continuous action of a locally compact group $G$ on $X$ admits stationary measures.}

\emph{But there also exist examples of continuous actions $G\curvearrowright X$ admitting stationary measures,
and neither $G$ nor $X$ is compact.
For example, let $G=SL(2,\Bbb R)$ and $X =  G/\Gamma$, where $\Gamma=SL(2,\Bbb Z)$,
and consider the action $G\curvearrowright X$. 
It is well-known that $\Gamma$ is a non-uniform lattice in $G$, i.e. the homogeneous space $X$ is non-compact 
and $X$ has a $G$-invariant probability measure.
Note that this action is transitive, hence both
ergodic and minimal.}
\end{remark}

\end{document}